\documentclass{amsart}
\usepackage{amsmath, amscd, amssymb, amsthm}
\usepackage{bbm}
\usepackage{latexsym}
\usepackage{amsfonts}
\usepackage{graphicx}
\usepackage[all,cmtip]{xy}

\usepackage[dvipdfm,
colorlinks=true, linkcolor=black,breaklinks=true,
urlcolor=blue, citecolor=black
]{hyperref}%
\usepackage{geometry}
\newtheorem{theorem}{Theorem}
\newtheorem{lemma}{Lemma}
\newtheorem{corollary}[theorem]{Corollary}

\newcommand{\be}{\begin{equation}}
\newcommand{\ee}{\end{equation}}
\newcommand{\bea}{\begin{eqnarray}}
\newcommand{\eea}{\end{eqnarray}}

\newcommand{\eps}{\varepsilon}

\newcommand{\vs}{\vspace{0.5cm}}

\newcommand{\vsv}{\vspace{0.12cm}}

\def\XXint#1#2#3{{\setbox0=\hbox{$#1{#2#3}{\int}$ }
\vcenter{\hbox{$#2#3$ }}\kern-.6\wd0}}

\begin{document}

\title[Complex nilmanifolds and K\"ahler-like connections]{Complex nilmanifolds and K\"ahler-like connections}

\author{Quanting Zhao} \thanks{The research of QZ is partially supported by NSFC with Grant No.11801205
and China Scholarship Council to Ohio State University. The research of FZ is partially supported by a Simons Collaboration Grant 355557.}
\address{Quanting Zhao. School of Mathematics and Statistics \& Hubei Key Laboratory of Mathematical Sciences, Central
China Normal University, Wuhan 430079, China.}

\email{{zhaoquanting@126.com; zhaoquanting@mail.ccnu.edu.cn}}

\author{Fangyang Zheng} \thanks{}
\address{Fangyang Zheng. Department of Mathematics, The Ohio State University, Columbus, OH 43210, USA}
\email{{zheng.31@osu.edu}}


\begin{abstract}
In this note, we analyze the question of when will a complex nilmanifold have K\"ahler-like Strominger  (also known as Bismut), Chern, or Riemannian  connection, in the sense that the curvature of the connection obeys all the symmetries of that of a K\"ahler metric. We give a classification in the first two cases and a partial description in the third case. It would be interesting to understand these questions for all Lie-Hermitian manifolds, namely, Lie groups equipped with a left invariant complex structure and a compatible left invariant metric.
\end{abstract}

\maketitle

\tableofcontents

\markleft{Quanting Zhao and Fangyang Zheng}
\markright{Complex nilmanifolds and K\"ahler-like connections}

\section{Introduction}

On a Hermitian manifold $(M^n,g)$, the concept of a metric connection $D$ being {\em K\"ahler-like} traces back to the pioneer work of Gray and others in the 1960s, where they considered various types of special  Hermitian and almost Hermitian metrics when the Riemannian curvature tensor satisfies various symmetry conditions. In \cite{YZ}, Bo Yang and the second named author followed their lead and introduced the term {\em K\"ahler-like} for the Riemannian and Chern connections. Angella, Otal, Ugarte, and Villacampa \cite{AOUV} generalized it to any metric connection on a Hermitian manifold.

For a metric connection $D$ on a Hermitian manifold $(M^n,g)$, its curvature $R^D$ is given by
$$  R^D(X,Y,Z,W) = \langle D_XD_YZ - D_YD_XZ- D_{[X,Y]}Z, \, W \rangle , $$
where $g(\, , \, ) = \langle \, , \, \rangle$ and $X$, $Y$, $Z$, $W$ are tangent vectors in $M^n$. $R^D$ is  skew-symmetric with respect to its first two positions by definition, and is skew-symmetric with respect to its last two positions since $Dg=0$. $D$ is said to be {\em K\"ahler-like,} if $R^D$ satisfies the symmetry conditions
$$ R^D(X,Y,Z,W) = R^D(Z,Y,X,W), \ \ \ R^D(X,Y,JZ,JW)= R^D(X,Y,Z,W), $$
for any tangent vectors $X$, $Y$, $Z$, $W$ in $M^n$. Note that the second condition is always satisfied when $DJ=0$. So for Hermitian connections (namely, those with $Dg=0$ and $DJ=0$), the K\"ahler-like condition simply means that the curvature is symmetric with respect to its first and third positions.

Given a Hermitian manifold $(M^n,g)$, we will denote by $\nabla$, $\nabla^c$, and $\nabla^s$ the Riemannian, Chern, and Strominger (aka Bismut or KT) connection, respectively. Note that $\nabla^s$ is the unique Hermitian connection on $M^n$ whose torsion is totally skew-symmetric. It appeared explicitly in \cite{Strominger} in 1986 (where he called it the H-connection), prior to Bismut's paper \cite{Bismut} which appeared in 1989. For that reason we think it would be more appropriate to call it Strominger connection\footnote{We would like to thank Professor Yau for pointing out this historic fact to us.}, and we shall do so from now on.

These three canonical connections coincide when $g$ is K\"ahler, and are mutually distinct when $g$ is not K\"ahler. For each of these connections, there are K\"ahler-like examples that are non-K\"ahler. Some necessary conditions were obtained, for instance, it was observed in \cite{YZ} that if $(M^n,g)$ is a compact Hermitian manifold that is either Riemannian K\"ahler-like or Chern K\"ahler-like, then the metric $g$ must be balanced. Also, it was conjectured by Angella, Otal, Ugarte, and Villacampa \cite{AOUV} and proved recently by the authors \cite{ZZ} that any Strominger K\"ahler-like manifold is plurisclosed (also known as SKT, or Strong K\"ahler with torsion). However, a full classification of such manifolds seems to be still far away.

In this note, we restrict ourselves to a very special type of Hermitian manifolds, namely, the complex nilmanifolds, and try to understand the K\"ahler-like conditions amongst such manifolds.

For the sake of simplicity, let us call $(G,J,g)$ a {\em Lie-Hermitian manifold} if $G$ is a Lie group, $J$ a left invariant complex structure on $G$, and $g$ a left invariant metric on $G$ compatible with $J$. It is a highly restrictive type of Hermitian manifolds in the sense that it is topologically parallelizable and also admits a flat connection $D$ that is Hermitian (namely, $DJ=Dg=0$). One can simply take a unitary frame of left invariant vector fields on $G$ and use it as the parallel frame to define the connection $D$. Note that $D$ is an Ambrose-Singer connection (in the sense that both its torsion and curvature are parallel under $D$).

It would certainly be a very interesting question and perhaps also a somewhat realistic goal to try to classify all Lie-Hermitian manifolds that are Riemannian, Chern, or Strominger K\"ahler-like. In this note, we will prove the following  statements which are partial answers in the special case when the Lie group $G$ is nilponent.

\begin{theorem} \label{nilBKL}
Let $(G,J,g)$ be a Lie-Hermitian manifold, namely, $G$ is a Lie group equipped with a left invariant complex structure $J$ and a compatible left invariant metric $g$. If $G$ is nilpotent, then $(G,J,g)$ is Strominger K\"ahler-like if and only if the Lie algebra ${\mathfrak g}$ of $G$ is the following type of at most $2$-step nilpotent Lie algebra:

There exists an orthonormal basis $\{ X_1, \ldots , X_s\}$ of ${\mathfrak n} =[{\mathfrak g}, {\mathfrak g}]$ and an orthonormal basis $\{ \eps_1, \ldots , \eps_{2n}\}$ of ${\mathfrak g}$ with $J\eps_i = \eps_{n+i}$ for each $1\leq i\leq n$, such that
$${\mathfrak n} + \!J {\mathfrak n}\, = \mbox{span} \{ \eps_{r+1}, \ldots , \eps_n; \, \eps_{n+r+1}, \ldots , \eps_{2n}\},$$
and positive constants $\lambda_1, \ldots , \lambda_s$ such that the only non-trivial Lie brackets under $\{ \eps\} $ are
$$ [\eps_i, \eps_{n+i} ] = \lambda_i X_i, \ \ \  1\leq  i \leq s .$$
Note that the nonnegative integer $s$ satisfies $n\!-\!r\leq s\leq \min\{ r, 2(n\!-\!r)\}$, and the complex structure $J$ is abelian. The K\"ahler case corresponds to $s=0$ and $r=n$.
\end{theorem}

We can write the above in the more familiar complex notations. Let $e_i = \frac{1}{\sqrt{2}}(\eps_i \!-\! \sqrt{-1}\eps_{n\!+\!i})$ be the unitary frame and $\varphi$ the dual coframe. The above theorem simply says that when $G$ is nilpotnet, the Lie-Hermitian manifold  $(G,J,g)$ is Strominger K\"ahler-like if and only if there exists an left invariant unitary coframe $\varphi$ and integers $0\leq s\leq r\leq n$ with $(n\!-\!r)\leq s\leq 2(n\!-\!r)$ such that
\begin{equation}
 d\varphi_i = 0, \ \ 1\leq i\leq r; \ \ \ \ \ d\varphi_{\alpha} = \sum_{i=1}^s Y_{i\alpha} \varphi_i \overline{\varphi}_i, \ \ r\!+\!1\leq \alpha \leq n,
 \end{equation}
where $r$ is exactly the complex dimension of the space of all $d$-closed left invariant $(1,0)$-forms and the constants $Y_{i\alpha}$ satisfies
\begin{equation}\label{eq:Yfinal}
 \sum_{\alpha=r+1}^n (Y_{i\alpha}  \overline{Y_{k\alpha} } +  \overline{Y_{i\alpha} } Y_{k\alpha} ) =0 \ \ \ \ \forall \ 1\leq i\neq k \leq s.
 \end{equation}
These constants are related to the orthonormal basis $\{ X_i\} $ of ${\mathfrak n}$ by
$$\lambda_iX_i = \sqrt{-1} \sum_{\alpha=r+1}^n (Y_{i\alpha} e_\alpha - \overline{Y_{i\alpha}}\,\overline{e_\alpha}) .$$

In low dimensions, one can express these constants more explicitly and write down the following ``normal forms":

\begin{corollary}\label{nilBKL_lowd}
Let $(G,J,g)$ be a nilpotent Lie-Hermitian manifold of complex dimension $n \leq 6$. It is Strominger K\"ahler-like if and only if it admits a left invariant unitary coframe $\varphi$ such that
\begin{enumerate}
\item when $n=2$, \[ \begin{cases}  d\varphi_1=0, \\ d\varphi_2=\lambda \,\varphi_1 \overline{\varphi}_1,  \end{cases} \qquad\qquad\ \]
\item when $n=3$, \[ \begin{cases} d\varphi_1=d\varphi_2=0, \\ d\varphi_3=\lambda \,\varphi_1 \overline{\varphi}_1 + ia\,\varphi_2 \overline{\varphi}_2,   \end{cases} \]
\item when $n=4$, either
\[ \begin{cases} d\varphi_1=d\varphi_2= d\varphi_3=0, \\ d\varphi_4=\lambda \,\varphi_1 \overline{\varphi}_1 + ia\, \varphi_2 \overline{\varphi}_2,  \end{cases} \]
 or
\[ \ \begin{cases} d\varphi_1=d\varphi_2= 0, \\ d\varphi_3=\lambda_1 \,\varphi_1 \overline{\varphi}_1  + ia\,\varphi_2 \overline{\varphi}_2,  \\
 d\varphi_4= \qquad\qquad\ \lambda_2\, \varphi_2 \overline{\varphi}_2 , \end{cases} \]
 \item when $n=5$,  either
\[ \quad\quad \begin{cases}  d\varphi_1=d\varphi_2= d\varphi_3=d\varphi_4=0, \\
d\varphi_5=\lambda \,\varphi_1 \overline{\varphi}_1 + ia\, \varphi_2 \overline{\varphi}_2 , \end{cases}\]
 or
\[ \qquad \qquad \quad\begin{cases} d\varphi_1=d\varphi_2= d\varphi_3=0, \\
d\varphi_4=\lambda_1 \,\varphi_1 \overline{\varphi}_1 + ia\, \varphi_2 \overline{\varphi}_2 + ib\, \varphi_3 \overline{\varphi}_3 ,  \\
d\varphi_5 =  \qquad   \lambda_2\, \varphi_2 \overline{\varphi}_2 + (ic\!-\!\frac{ab}{\lambda_2})\, \varphi_3 \overline{\varphi}_3,
\end{cases} \]
 \item when $n=6$, either
\[ \qquad\qquad\quad\begin{cases} d\varphi_1=d\varphi_2= d\varphi_3=d\varphi_4= d\varphi_5= 0, \\
d\varphi_6=\lambda \,\varphi_1 \overline{\varphi}_1 + ia\, \varphi_2 \overline{\varphi}_2 ,  \end{cases}\]
 or
\[ \qquad\qquad\qquad\qquad\qquad\begin{cases} d\varphi_1=d\varphi_2=d\varphi_3=d\varphi_4= 0, \\
d\varphi_5 = \lambda_1 \,\varphi_1 \overline{\varphi}_1  + ia\, \varphi_2 \overline{\varphi}_2 + ib\, \varphi_3 \overline{\varphi}_3  + ic\, \varphi_4 \overline{\varphi}_4 ,  \\
d\varphi_6=  \lambda_2\, \varphi_2 \overline{\varphi}_2 +(ix\!-\!\frac{ab}{\lambda_2}) \, \varphi_3 \overline{\varphi}_3 + (iy\!-\!\frac{ac}{\lambda_2})\, \varphi_4 \overline{\varphi}_4   , \end{cases}\]
 or
\[ \qquad \qquad \qquad \quad\begin{cases} d\varphi_1=d\varphi_2=d\varphi_3= 0, \\
d\varphi_4 =\lambda_1 \,\varphi_1 \overline{\varphi}_1 + ia\, \varphi_2 \overline{\varphi}_2 + ib\, \varphi_3 \overline{\varphi}_3 ,  \\
d\varphi_5 = \ \ \ \ \ \ \ \ \ \ \ \ \ \lambda_2 \,\varphi_2 \overline{\varphi}_2 + (ic\!-\!\frac{ab}{\lambda_2}) \varphi_3 \overline{\varphi}_3,\\
d\varphi_6=  \ \ \ \ \ \ \ \ \ \ \ \ \ \ \ \ \ \ \ \ \ \ \ \ \ \ \lambda_3\, \varphi_3 \overline{\varphi}_3 ,  \end{cases}\]
\end{enumerate}
where the constants $a$, $b$, $c$, $x$, $y$, $\lambda $, $\lambda_1$, $\lambda_2$,  $\lambda_3$ are all real, with $\lambda \geq 0$, $\lambda_1>0$, $\lambda_2>0$, $\lambda_3>0$.
\end{corollary}

The $n=2$ case (when $\lambda >0$) corresponds to the Kodaira surfaces. The $n=3$ case was obtained by Angella, Otal, Ugarte, and Villacampa \cite{AOUV}. They examined all $t$-Gauduchon K\"ahler-like manifolds amongst all nilmanifolds and Calabi-Yau type solvmanifolds with $n=3$.

For the Chern K\"ahler-like condition, we have the following

\begin{theorem} \label{nilCKL}
Let $(G,J,g)$ be a Lie-Hermitian manifold, namely, $G$ is a Lie group equipped with a left invariant complex structure $J$ and a compatible left invariant metric $g$. If $G$ is a nilpotent, then $(G,J,g)$ is Chern K\"ahler-like if and only if it is Chern flat, namely, $(G,J)$ is a complex Lie group.
\end{theorem}

The $n=3$ case of this result was proved by Angella, Otal, Ugarte and Villacampa \cite{AOUV}, where they examined all nilmanifolds and Calabi-Yau type solvmanifolds of the complex dimension $n=3$ and concluded that all Chern K\"ahler-like examples there are actually Chern flat.

Of course we do believe that there should be examples of compact Chern K\"ahler-like manifolds that are not Chern flat, perhaps even among (compact quotients of) Lie-Hermitian manifolds, but so far we do not know of any such example yet.

Similarly, for the Riemannian K\"ahler-like condition, we have the following negative statement

\begin{theorem} \label{nilRKL}
Let $(G,J,g)$ be a Lie-Hermitian manifold, namely, $G$ is a Lie group, $J$ a left invariant complex structure on $G$, and $g$ a left invariant metric on $G$ compatible with $J$. If $J$ is a nilpotent complex structure in the sense of Cordero, Fern\'{a}ndez, Gray and Ugarte \cite{CFGU}, then $(G,J,g)$ is Riemannian  K\"ahler-like if and only if $G$ is abelian.
\end{theorem}

Note that when $J$ is nilpotent, the Lie group $G$ must be nilpotent as in \cite[Theorem 13]{CFGU}, but the converse is not true, namely, there are examples of nilpotent $G$ with non-nilpotent complex structure $J$. We do believe that the above statement should still be true under the broader assumption that $G$ is nilpotent, but at this point we do not know how to prove it, and thus we settle with the stronger assumption that $J$ is nilpotent.

%
%

\section{The Strominger K\"ahler-like case}\label{BKL}

Let us start with a Lie-Hermitian manifold $(G,J,g)$. We will follow the notations of \cite{VYZ}. Let $e$ be a unitary frame of left invariant  vector fields of type $(1,0)$ on $G$, with $\varphi$ the dual coframe. We will also denote by $e$ the corresponding unitary basis of ${\mathfrak g}^{1,0}$, the $(1,0)$-part of ${\mathfrak g}^{\mathbb C}$, which is the comlexification of the Lie algebra ${\mathfrak g}$ of $G$.

Extend the inner product  $g( ,)=\langle \, , \rangle$ linearly on ${\mathfrak g}^{\mathbb C}$, and denote by
\begin{equation} \label{str-const}
C_{ik}^j = \langle [e_i, e_k], \overline{e}_j \rangle , \ \ \  D_{ik}^j = \langle [\overline{e}_j, e_k], e_i \rangle
\end{equation}
the structure constants. Then the Chern torsion components and the Strominger connection coefficients are
\begin{eqnarray} \label{str-const}
& & 2T_{ik}^j = -C_{ik}^j -D_{ik}^j + D_{ki}^j , \\
& &   \Gamma_{ik}^j := \langle \nabla^s_{e_k}e_i ,\overline{e}_j \rangle = D_{ik}^j + 2T_{ik}^j = -C_{ik}^j + D_{ki}^j.
\end{eqnarray}
The covariant differentiation of the torsion with respect to $\nabla^s$ are given by
\begin{eqnarray}
& & T_{ik,\ell }^j = \sum_{p=1}^n (-T_{pk}^j \Gamma^p_{i\ell } - T_{ip}^j \Gamma^p_{k\ell } + T_{ik}^p \Gamma^j_{p\ell }), \label{eq:deri} \\
& &  T_{ik,\overline{\ell }}^j = \sum_{p=1}^n (T_{pk}^j \overline{ \Gamma^i_{p\ell } } + T_{ip}^j \overline{ \Gamma^k_{p\ell } } - T_{ik}^p \overline{ \Gamma^p_{j\ell } } ), \label{eq:deribar}
\end{eqnarray}
and the structure equation is given by
\begin{equation} \label{str}
d \varphi_j = - \frac{1}{2} \sum_{i,k=1}^n C^j_{ik} \varphi_i \wedge \varphi_k - \sum_{i,k=1}^n \overline{D^i_{jk}} \, \varphi_i \wedge \overline{\varphi}_k .
\end{equation}
The following result of Enrietti, Fino, and Vezzoni \cite{EFV} will be a crucial starting point for the proof of Theorem \ref{nilBKL} (see also \cite[Lemma 2.2]{FV}):

\begin{lemma}[\cite{EFV}]
If $G$ is nilpotent and $g$ is pluriclosed, then there exists a left invariant unitary coframe $\varphi$  and an integer $1\leq r\leq n$ such that
\begin{eqnarray} \label{strnil}
& & d\varphi_i = 0, \ \ \ 1\leq i \leq r ;  \\
& & d \varphi_{\alpha} = - \frac{1}{2} \sum_{i,k=1}^r C^{\alpha}_{ik} \varphi_i \wedge \varphi_k - \sum_{i,k=1}^r \overline{D^i_{\alpha k}} \, \varphi_i \wedge \overline{\varphi}_k, \ \ \ r\!+\!1 \leq \alpha \leq n.
\end{eqnarray}
\end{lemma}

Throughout this section, we will use the index range convention that
$$ 1\leq i, j, k, \ldots  \leq r \ \ \ \  \mbox{and} \ \ \ \ r\!+\!1 \leq \alpha, \beta, \gamma , \ldots  \leq n. $$
The above lemma says that, when $G$ is nilpotent and $g$ is pluriclosed, the only possibly non-zero components of $C$ and $D$ are $C^{\alpha }_{ik}$ and $D^i_{\alpha k}$. From this, we know that the only possibly non-zero components of $T$ and $\Gamma$ are
$$ 2T^{\alpha}_{ik} = \Gamma^{\alpha }_{ik} = - C^{\alpha }_{ik}, \ \ \ 2T^i_{k\alpha } = \Gamma^i_{k\alpha } = D^i_{\alpha k}, $$
while $\Gamma^{i }_{\alpha k} = 0$. Plugging this into the derivative formula (\ref{eq:deribar}), we get the following

\begin{lemma} \label{abelian}
If $G$ is nilpotent, $g$ is pluriclosed, and $\nabla^sT=0$, then $C=0$, namely, the complex structure $J$ is abelian.
\end{lemma}

\begin{proof}
Since $T^{\alpha}_{\alpha k}=0$ and $\Gamma^{\ast}_{\alpha k}=0$, by the derivative formula (\ref{eq:deribar}), we get
$$ 0 =  T^{\alpha }_{\alpha k, \overline{k} } = \sum_p T^{\alpha}_{pk}  \overline{ \Gamma^{\alpha}_{pk}  } = \frac{1}{2} \sum_p | C^{\alpha}_{pk} |^2 $$
for any $\alpha $ and $k$, Therefore $C=0$.
\end{proof}

By this lemma we have $C=0$, so now the only possibly non-zero components of $T$ and $\Gamma$ are
$$ 2T^i_{k\alpha } = \Gamma^i_{k\alpha } = D^i_{\alpha k}.$$
From the derivative formula (\ref{eq:deri}) and (\ref{eq:deribar}), it is easy to establish the following

\begin{lemma} \label{Dcommute}
If $G$ is nilpotent and $g$ is pluriclosed, then $\nabla^sT=0$ if and only if $C=0$ and
$$ D_{\alpha}D_{\beta} = D_{\beta}D_{\alpha}, \ \ D_{\alpha}^{\ast }D_{\beta} = D_{\beta}D_{\alpha}^{\ast} $$
for any $\alpha$, $\beta$.
\end{lemma}

Here we denoted by $D_{\alpha}$ the $r\times r$ matrix, whose $(i,j)$-th entry is $D^j_{\alpha i}$. The above lemma says that these $D_{\alpha }$ form a set of commuting normal matrices, hence they can be simultaneously diagonalized by unitary matrices, namely, there exists a unitary matrix $P\in U(r)$ such that $UD_{\alpha}U^{\ast }$ is diagonal for each $\alpha$.

\begin{proof}[\textbf{Proof of Theorem \ref{nilBKL}}] Let $(G,J,g)$ be a Lie-Hermitian manifold that is Strominger K\"ahler-like and assume that $G$ is nilpotent. Then by \cite[Corollary 4]{ZZ}, we know that $g$ is pluriclosed and $\nabla^sT=0$. Hence the above lemmata imply that there exists a unitary left invariant coframe $\varphi$ and an integer $1\leq r \leq n$ such that
\begin{eqnarray}
& & d\varphi_i = 0, \ \ \ 1\leq i \leq r ;  \\
& & d \varphi_{\alpha} =  \sum_{i=1}^r Y_{i\alpha}  \varphi_i \wedge \overline{\varphi}_i, \ \ \ r\!+\!1 \leq \alpha \leq n.
\end{eqnarray}
Note that here we have performed a unitary change on $\{ \varphi_1, \ldots , \varphi_r\}$ if necessary, to ensure that all $D_{\alpha}$ are diagonal: $D^k_{\alpha i} = - \overline{Y_{i\alpha } }\,\delta_{ik}$.

Let us denote by $\xi_{\alpha}$ the column vector with entries $Y_{i\alpha }$, $1\leq i \leq r$. By performing a unitary change of $\{ \varphi_{r+1}, \ldots , \varphi_n\}$ if necessary, we may assume that the collection $\{ \xi_{\tilde{r}+1}, \ldots , \xi_n\}$ is linearly independent for some integer $\tilde{r}$ possibly larger than $r$ and $\xi_{\alpha}=0$ for $r+1 \leq \alpha \leq \tilde{r}$. This means that the $r\times (n-\tilde{r})$ matrix $(Y_{i\alpha})$ has rank $n-\tilde{r}$. Note that now $\tilde{r}$ stands for the complex dimension of the space of all $d$-closed left invariant $(1,0)$-forms on $G$ and the structure equation amounts to
\begin{equation}
 d\varphi_i = 0, \ \ 1\leq i\leq \tilde{r}; \ \ \ \ \ d\varphi_{\alpha} = \sum_{i=1}^r Y_{i\alpha} \varphi_i \overline{\varphi}_i, \ \ \tilde{r}\!+\!1\leq \alpha \leq n.
\end{equation}

Since $\partial \varphi_{\alpha }=0$ for $ \tilde{r}+1 \leq \alpha \leq n$, we get $\overline{\partial}(\varphi_{\alpha}\overline{\varphi}_{\alpha}) = \overline{\partial}\varphi_{\alpha} \wedge \overline{\varphi}_{\alpha}$ and
\begin{eqnarray*}
\partial \overline{\partial} (\varphi_{\alpha} \overline{\varphi}_{\alpha} ) & = & \partial (\overline{\partial} \varphi_{\alpha} \wedge \overline{\varphi}_{\alpha}) \ = \ \overline{\partial} \varphi_{\alpha} \wedge \partial \overline{\varphi}_{\alpha} \\
& = &  - \sum_{i,k=1}^r Y_{i\alpha} \overline{ Y_{k\alpha }} \, \varphi_i\, \overline{\varphi}_i \, \varphi_k \,  \overline{\varphi}_k.
\end{eqnarray*}
Therefore, the pluriclosed condition $\partial \overline{\partial}  \omega_g =0$ is equivalent to
\begin{equation} \label{eq:Y}
\sum_{\alpha =\tilde{r}+1}^n ( Y_{i\alpha} \overline{ Y_{k\alpha} }
+ Y_{k\alpha} \overline{ Y_{i\alpha} } ) = 0, \ \ \ \forall \ 1\leq i \neq  k\leq r.
\end{equation}
Denote by $y_i$ the vector in ${\mathbb C}^{n-\tilde{r}}$ whose entries are $Y_{i\alpha}$. We have
$$ \langle y_i , \overline{y_k} \rangle + \langle y_k , \overline{y_i} \rangle =0 , \ \ \ \forall \ 1\leq i \neq  k\leq r. $$
Let us write $Y_{i\alpha} = U_{i\alpha} + \sqrt{-1} V_{i\alpha}$ for the real and imaginary parts, and write $u_i= (U_{i\alpha})$, $v_i = (V_{i\alpha})$ for the vectors in ${\mathbb R}^{n-\tilde{r}}$. It follows that $y_i=u_i+\sqrt{-1}v_i$. Let us also denote by $x_i = (- v_i, u_i)$ the vector in ${\mathbb R}^{2(n-\tilde{r})}$. The above condition on $y_i$ when translated in terms of $x_i$ becomes
$$ \langle x_i, x_k \rangle =0 , \ \ \ \forall \ 1\leq i \neq  k\leq r,  $$
that is, the vectors $\{ x_1, \ldots , x_r\} $ in ${\mathbb R}^{2(n-\tilde{r})}$ are pairwisely perpendicular to each other. Let $s$ be the number of non-zero $x_i$, and by a permutation if necessary, we may assume that $x_i\neq 0$ for each $1\leq i \leq s$, while $x_i=0$ for each $s+1\leq i\leq r$. Clearly, $s\leq 2(n-\tilde{r})$ and we also know that $s\geq n-\tilde{r}$ as there are only $s$ rows in $(Y_{i\alpha})$ that are non-zero, whereas the matrix has rank $n-\tilde{r}$. Hence we know that $s$ is in the range
$$ n\!-\!\tilde{r} \leq s\leq \min \{ r, 2(n\!-\!\tilde{r})\} \leq \min \{ \tilde{r}, 2(n\!-\!\tilde{r})\} ,$$
and $\{ x_1, \ldots , x_s\}$ form an orthogonal basis for a $s$-dimensional subspace of ${\mathbb R}^{2(n-\tilde{r})}$.

Then let us express things in terms of vector fields. Let $e$ be the left invariant unitary frame dual to $\varphi$. Under the frame $\{ e_a, \overline{e}_a\}_{a=1}^{2n}$ of ${\mathfrak g}^{\mathbb C}$, the only non-trivial Lie brackets are
 $$ [\overline{e}_i, e_i ] = \sum_{\alpha=\tilde{r}+1}^n (Y_{i\alpha} e_{\alpha} - \overline{ Y_{i\alpha} } \,\overline{e}_{\alpha}), \ \ \ \ \ 1\leq i\leq r. $$
When we write $e_a = \frac{1}{\sqrt{2}} (\eps_a - \sqrt{-1} \eps_{n+a})$, it yields that, under the basis $\{ \eps_a, \eps_{n+a}\} _{a=1}^n$ of ${\mathfrak g}$, the only non-zero brackets are
$$ [\eps_i, \eps_{n+i} ] = \sqrt{-1} \,[\overline{e}_i, e_i ] =  \sqrt{2} \sum_{\alpha =\tilde{r}+1}^n ( U_{i\alpha} \eps_{n+\alpha} - V_{i\alpha} \eps_{\alpha} ) = \sqrt{2} \,x_i, \ \ \ \ 1\leq i\leq s, $$
since $x_i=0$ for $s+1\leq i\leq r$. If we normalize $\sqrt{2}x_i = \lambda_i X_i$ with $\lambda_i>0$ and $|X_i|=1$, then $\{ X_1, \ldots , X_s\}$ form an orthonormal basis of ${\mathfrak n} = [{\mathfrak g}, {\mathfrak g} ]$.

Conversely, if a Lie-Hermitian manifold $(G,J,g)$ satisfies the description in Theorem \ref{nilBKL}. It is easy to see that it is Strominger K\"ahler-like. This completes the proof of Theorem \ref{nilBKL}.
\end{proof}

As to the proof of Corollary \ref{nilBKL_lowd}, since it is an easy consequence of the condition (\ref{eq:Yfinal}) and the above considerations, therefore we omit it here.

\section{The Chern and Riemannian K\"ahler-like cases}\label{BKL}

In this section, we would like to prove Theorem \ref{nilCKL} and \ref{nilRKL} stated in the introduction. Let $(G,J,g)$ be a Lie-Hermitian manifold and $e$ be a left invariant unitary frame with dual coframe $\varphi$ as before. The coefficients of the Chern connection $\nabla^c$ under  $e$ are given by
\begin{equation}
\hat{\Gamma}_{ik}^{j} := \langle \nabla^c_{e_k}e_i ,\overline{e}_j \rangle = D_{ik}^j ,
\end{equation}
hence the covariant derivatives of the torsion in $\nabla^c$ become
\begin{eqnarray}
& & T_{ik;\ell }^j = \sum_{r=1}^n (-T_{rk}^j D^r_{i\ell } - T_{ir}^j D^r_{k\ell } + T_{ik}^r D^j_{r\ell }), \label{eq:cderi} \\
& &  T_{ik;\overline{\ell }}^j = \sum_{r=1}^n (T_{rk}^j \overline{ D^i_{r\ell } } + T_{ir}^j \overline{ D^k_{r\ell } } - T_{ik}^r \overline{ D^r_{j\ell } } ). \label{eq:cderibar}
\end{eqnarray}

The Jacobi identity, or equivalently, the exterior differentiation of the structure equation (\ref{str}), gives the following identities for the structure constants $C$ and $D$ (cf. \cite[Lemma 2.1]{VYZ}):
\begin{eqnarray}
&& \sum_{r=1}^n \big( C^r_{ij}C^{\ell}_{rk} + C^r_{jk}C^{\ell}_{ri} + C^r_{ki}C^{\ell}_{rj} \big) \ =\ 0, \\
 && \sum_{r=1}^n \big(  C^r_{ik} D^{\ell}_{j r} + D^r_{ji} D^{\ell}_{rk} - D^r_{jk} D^{\ell}_{ri}   \big) \ =\ 0, \\
 && \sum_{r=1}^n \big(   C^r_{ik} \overline{D^r_{j\ell }} - C^{j}_{rk} \overline{D^i_{r\ell}} + C^{j}_{ri} \overline{D^k_{r \ell}} - D^{\ell}_{ri} \overline{D^k_{j r}} + D^{\ell}_{rk} \overline{D^i_{j r}}  \big) \ =\ 0 \label{eq:CDbar}
\end{eqnarray}
for any indices $i$, $j$, $k$, and $\ell$. Since $2T_{ik}^j = - C_{ik}^j - D_{ik}^j  + D_{ki}^j$, by (\ref{eq:cderibar}) and (\ref{eq:CDbar}), we get
\begin{eqnarray}
2T_{ik;\overline{\ell }}^j & = &  -\big(  C_{rk}^j + D_{rk}^j  - D_{kr}^j \big) \overline{ D^i_{r\ell } } - \big( C_{ir}^j + D_{ir}^j  - D_{ri}^j \big) \overline{ D^k_{r\ell } } + \big(  C_{ik}^r + D_{ik}^r  - D_{ki}^r \big)  \overline{ D^r_{j\ell } } \nonumber \\
& = & -\big(D_{rk}^j - D_{kr}^j \big) \overline{ D^i_{r\ell } } - \big(D_{ir}^j - D_{ri}^j\big) \overline{ D^k_{r\ell } } + \big(  D_{ik}^r  - D_{ki}^r \big)  \overline{ D^r_{j\ell } } + \big( D^{\ell }_{ri} \overline{ D^k_{jr} }  -  D^{\ell }_{rk} \overline{ D^i_{jr} }  \big). \label{eq:D8}
\end{eqnarray}

\begin{proof}[\textbf{Proof of Theorem \ref{nilCKL}}] Suppose that $(G,J,g)$ is  a Lie-Hermitian manifold, and assume that $G$ is a nilpotent. By the famous result of Salamon \cite[Theorem 1.3]{Salamon}, there will be a coframe $\varphi$ of left invariant $(1,0)$-forms on $G$ such that
$$ d\varphi_1 =0, \ \ \ d\varphi_i = {\mathcal I} \{\varphi_1, \ldots , \varphi_{i-1}\} , \ \ \ \forall \ 2\leq i\leq n, $$
where ${\mathcal I}$ stands for the ideal in  exterior algebra of the complexified cotangent bundle generated by those $(1,0)$-forms. Clearly, one can assume that $\varphi$ is also unitary. In terms of the structure constants $C$ and $D$, this means that
\begin{equation}
C^j_{ik}=0  \ \ \ \mbox{unless} \ \ j>i \ \mbox{or} \ j>k; \ \ \ \ \ D^i_{jk}=0  \ \ \ \mbox{unless} \ \ j>i.  \label{Salamon}
\end{equation}
Assume that $g$ is Chern K\"ahler-like. By \cite{YZ} we know that this is equivalent to the condition that $T_{ik;\overline{\ell }}^j=0$ for any indices $i$, $j$, $k$, and $\ell$. Hence the sum of the terms in the line (\ref{eq:D8}) vanishes. If we take $j=k=n$ there, since $D^n_{\ast \ast }=0$, it yields that
$$  \sum_r \big\{ \big( D^r_{in} - D^r_{ni}\big) \overline{ D^r_{n\ell} } - D^{\ell}_{rn}  \overline{ D^i_{nr} } \big\} =0 $$
for any $i$, $\ell$. Taking $i=\ell$ and summing over, we get
$$ \sum_{i,r} |D^r_{ni}|^2 = \sum_{i,r} D^r_{in} \overline{D^r_{ni} } - \sum_{i,r} D^i_{rn} \overline{D^i_{nr} } =0.$$
Therefore we conclude that $D^{\ast}_{n\ast}=0$. To see that $D^{\ast}_{\ast n}=0$, let us take $k=\ell =n$ in (\ref{eq:D8}) and use the fact that $D^n_{\ast\ast} =D^{\ast}_{n\ast}=0$, which yields that
$$ \sum_r \big\{ - D^j_{rn} \overline{ D^i_{rn} } + D^r_{in} \overline{ D^r_{jn} } \big\} =0. $$
If we write $P=(P_{ij}) = (D^j_{in})$, the above equation simply says that $P^{\ast}P = P P^{\ast}$, that is, $P$ is normal hence diagonalizable. But by (\ref{Salamon}), $P$ is strictly lower triangular, hence nilpotent and all of its eigenvalues are zero. This means that $P=0$, so we have shown that $D$ vanishes whenever any of the indices is $n$. Now repeat the argument by taking $j=k=n-1$ in (\ref{eq:D8}) etc., from which we get $D$ vanishes whenever any of its indices is $n-1$. Keep on going with this process and we see that $D=0$ for all indices. Therefore $(G,J)$ is a complex Lie group and $g$ is Chern flat. This concludes the proof of Theorem \ref{nilCKL}.
\end{proof}


Let us now turn our attention to the Riemannian K\"ahler-like case. In this case the easiest covariant derivatives to use are those with respect to the so-called $0$-Gauduchon connection $\nabla^0$, defined by
$$ \nabla^0_XY = \frac{1}{2}(\nabla_XY - J\nabla_XJY)$$
for any tangent vector fields $X$ and $Y$. It is just the Hermitian projection of the Riemannian (Levi-Civita) connection $\nabla$, and is also called the Lichnerowicz connection in some literature. It is easy to see that $\nabla^0 = \frac{1}{2}(\nabla^c + \nabla^s)$.  In the remainder of this section, we will denote by $T^j_{ik,\ell}$ and $T^j_{ik, \overline{\ell}}$ the covariant derivatives of the Chern torsion under the connection $\nabla^0$. First we have the following

\begin{lemma}
If a Hermitian manifold $(M^n,g)$ is Riemannian K\"ahler-like, then under any local unitary frame $e$ it holds
\begin{eqnarray}
T^j_{ik,\ell} & = & \sum_{r=1}^n T^r_{ik} T^j_{r\ell}, \\
T^j_{ik, \overline{\ell}} & = & T^{\ell }_{ik, \overline{j}}, \label{eq:Tbar}
\end{eqnarray}
for any indices, where the indices after comma stands for covariant derivatives with respect to the $0$-Gauduchon connection $\nabla^0$.
\end{lemma}

\begin{proof} Let us use the notation of \cite{YZ} and denote by $\nabla e = \theta_1 e + \overline{\theta_2} \overline{e}$ the connection matrices of $\nabla$ under the frame $\{ e, \overline{e}\}$. Then $\nabla^0e = \theta_1 e$ as $Je=\sqrt{-1}e$ and $J\overline{e} = - \sqrt{-1}\overline{e}$. We have
$$ \theta_1 = \theta + \gamma , \ \ \ \gamma_{ij} = T^j_{ik} \varphi_k - \overline{T^i_{jk}} \overline{\varphi}_k, \ \ \  (\theta_2)_{ij} = \overline{ T^k_{ij} } \varphi_k, $$
where $\varphi$ is the unitary coframe dual to $e$ and $\theta $ the connection matrix of $\nabla^c$ under $e$.

By \cite[Lemma 5]{YZ}, we know that $g$ is Riemannian K\"ahler-like if and only if
$$ \Theta_2 := d\theta_2 - \theta_2 \theta_1 - \overline{\theta}_1 \theta_2 = 0.$$
Let us fix a point $x\in M$ and take a local unitary frame $e$ near $x$ so that $\theta_1$ vanishes at $x$. Note that this can always be managed for any given Hermitian connection by the same proof of \cite[Lemma 4]{YZ}. So at the point $x$, we have $\theta = - \gamma$, and by the structure equation together with the fact $\,^t\!\gamma'\varphi = - \tau$,  it yields that
$$ d\varphi = - \,^t\!\theta \varphi + \tau = \,^t\!\gamma \varphi + \tau = - \overline{\gamma'}\varphi. $$
Here $\gamma'$ is the $(1,0)$-part of $\gamma$. So at the point $x$, it follows that $\partial \varphi =0$ and $\overline{\partial} \varphi =- \overline{\gamma'}\varphi$. At $x$, we compute that
\begin{eqnarray*}
0 & = & (\Theta_2)_{ik} \ = \ d ((\theta_2)_{ik}) \ = \ \partial \,(\overline{T^j_{ik}} \varphi_j) + \overline{ \partial}\, (\overline{T^j_{ik}} \varphi_j)\\
& = & \overline{  T^j_{  ik, \overline{\ell}   }    } \,\varphi_{\ell} \varphi_j + \overline{  T^j_{  ik, \ell   }    }\, \overline{\varphi}_{\ell} \varphi_j - \overline{  T^j_{  ik} } \, \overline{ T^p_{jq}} \, \overline{\varphi}_q \varphi_p \\
& = & \frac{1}{2} \big( \overline{  T^j_{ik, \overline{\ell}}   - T^{\ell}_{  ik, \overline{j} }    } \big) \varphi_{\ell} \varphi_j + \big( \overline{    T^j_{ik, \ell} - T^r_{ik}   T^j_{r\ell }   }   \big) \overline{\varphi}_{\ell} \varphi_j,
\end{eqnarray*}
which establishes the two identities in the lemma.
\end{proof}

Let us specialize to a Lie-Hermitian manifold $(G,J,g)$. Let $e$ be an left invariant unitary frame, with dual coframe $\varphi$. It yields that
$$ 2T^j_{ik} = - C^j_{ik} - D^j_{ik} + D^j_{ki},  \ \ \ \Gamma^j_{ik} = D^j_{ik} + T^j_{ik} ,$$
where $C$, $D$ are structure constants as before, but $\Gamma$ now stands for the coefficients for the the connection $\nabla^0$, which is an abuse of notation as the symbol $\Gamma$ already appeared in the previous section. Note that one could simply write $ 2T = -C -D + \,^t\!D$ and $2\Gamma = -C + D + \,^t\!D$. It follows that
$$ T^j_{ik, \overline{\ell}} = \sum_{r=1}^n \big( T^j_{rk} \overline{\Gamma^i_{r\ell }} + T^j_{ir} \overline{\Gamma^k_{r\ell }} - T^r_{ik} \overline{\Gamma^r_{j\ell }} \big) .
$$
Plugging this into (\ref{eq:Tbar}) and using the fact that $\Gamma - \,^t\!\Gamma = -C$, we obtain
\begin{equation}
\sum_{r=1}^n \big( T^j_{rk} \overline{\Gamma^i_{r\ell }} + T^j_{ir} \overline{\Gamma^k_{r\ell }} + T^r_{ik} \overline{C^r_{j\ell }} - T^{\ell }_{rk} \overline{\Gamma^i_{rj }}  - T^{\ell }_{ir} \overline{\Gamma^k_{rj }}  \big) = 0. \label{eq:TCbar}
\end{equation}

\begin{proof}[\textbf{Proof of Theorem \ref{nilRKL}}]
Let us take advantage of the fact that $J$ is nilpotent. By \cite[Theorem 12]{CFGU}, there will be a left invariant unitary frame $e$, under which
\begin{equation}
 C^j_{ik} = D^i_{jk} = 0    \ \ \ \mbox{unless} \ \ j > i,k.  \label{CDnil}
 \end{equation}
In particular, $T^n_{\ast n}= \Gamma^n_{\ast n} =0$. Hence if we let $k=\ell =n$ in (\ref{eq:TCbar}), it follows that
$$ \sum_{r=1}^n \big(  T^j_{rn} \overline{\Gamma^i_{rn }}  -  T^{n}_{ir} \overline{\Gamma^n_{rj }}  \big) = 0. $$
That is, $ \sum_r D^j_{nr} \overline{D^i_{nr}} = - \sum_r C^n_{ir} \overline{C^n_{jr} }$. If we denote by $C^n$ the matrix with $(i,j)$-th entry $C^n_{ij}$ and by $D_n$ the matrix with $(i,j)$-th entry $D^j_{ni}$, then this means that
$$ D_n^{\ast } D_n + C^n (C^n)^{\ast} =0 .$$
Both terms are semi-positive definite Hermitian matrices and we conclude from the above identity that $C^n=D_n=0$, which implies that $C$ and $D$ will be zero whenever any of the indices is $n$. The  condition (\ref{CDnil}) on $C$ and $D$ now indicates that $C^{\ast}_{n-1,\ast} =0$, $D^{n-1}_{\ast \ast} = D^{\ast }_{\ast,n-1}=0$. Therefore we can take $k=\ell =n-1$ in (\ref{eq:TCbar}) and repeat the argument, which leads to the conclusion that $C=D=0$ whenever any of the indices is $n-1$. Keep on going, and thus we conclude in the end that $C=D=0$, which implies $G$ is abelian. This completes the proof of Theorem \ref{nilRKL}.
\end{proof}

\vsv
\vsv
\vsv
\vsv

\noindent\textbf{Acknowledgments.} The first named author is grateful to the Mathematics Department of Ohio State University for the nice research environment and the warm hospitality during his stay. The second named author would like to thank his collaborators Luigi Vezzoni, Qingsong Wang, and Bo Yang in their previous works, which laid the foundation for the computation carried out in the present paper.

\vs

\end{document}